\documentclass[11pt]{amsart}   
\usepackage{amssymb,amstext,amsmath,amscd,amsthm,amsfonts,enumerate,graphicx,latexsym,accents}
\usepackage{array}
\usepackage{longtable}
\usepackage{booktabs}
\usepackage{mathdots}
\usepackage{graphicx}
\usepackage[pagebackref,colorlinks=true,linkcolor=blue,urlcolor=blue]{hyperref}
\usepackage{stmaryrd}
\usepackage[initials, lite]{amsrefs}
\usepackage{color}
\usepackage[onehalfspacing]{setspace}
\usepackage{tabularx}
\usepackage{amsfonts} 
\usepackage{paralist}
\usepackage{aliascnt}
\usepackage{amscd}
\usepackage{blkarray}
\usepackage{mathbbol}
\usepackage{setspace}
\usepackage[inner=3cm,outer=3cm, bottom=3.2cm]{geometry}
\usepackage{tikz, tikz-cd}
\usepackage{amsxtra, epsfig,epic,graphics}
\usepackage{tikz}
\usetikzlibrary{matrix}
\usetikzlibrary{cd}
\usetikzlibrary{arrows,calc}
\allowdisplaybreaks

\BibSpec{collection.article}{%
	+{}  {\PrintAuthors}                {author}
	+{,} { \textit}                     {title}
	+{.} { }                            {part}
	+{:} { \textit}                     {subtitle}
	+{,} { \PrintContributions}         {contribution}
	+{,} { \PrintConference}            {conference}
	+{}  {\PrintBook}                   {book}
	+{,} { }                            {booktitle}
	+{,} { }                            {series}
	+{, vol.} { }                            {volume}
	+{,} { }                            {publisher}
	+{,} { \PrintDateB}                 {date}
	+{,} { pp.~}                        {pages}
	+{,} { }                            {status}
	+{,} { \PrintDOI}                   {doi}
	+{,} { available at \eprint}        {eprint}
	+{}  { \parenthesize}               {language}
	+{}  { \PrintTranslation}           {translation}
	+{;} { \PrintReprint}               {reprint}
	+{.} { }                            {note}
	+{.} {}                             {transition}
	+{}  {\SentenceSpace \PrintReviews} {review}
}
\AtBeginDocument{%
	\def\MR#1{}
}

\newcommand{\m}{\mathfrak{m}}

\newcommand{\reg}{\normalfont\text{reg}}

\newcommand{\Tor}{\normalfont\text{Tor}}

\newtheorem{theorem}{Theorem}[section]

\newaliascnt{headcor}{headthm}

\aliascntresetthe{headcor}

\newaliascnt{headconj}{headthm}

\aliascntresetthe{headconj}

\newaliascnt{corollary}{theorem}
\newtheorem{corollary}[corollary]{Corollary}
\aliascntresetthe{corollary}

\def\min{\operatorname{min}}

\def\reg{\operatorname{reg}}

\def\dim{\operatorname{dim}}

\def\m{\mathfrak m}

\newaliascnt{lemma}{theorem}
\newtheorem{lemma}[lemma]{Lemma}
\aliascntresetthe{lemma}

\newaliascnt{conjecture}{theorem}
\newtheorem{conjecture}[conjecture]{Conjecture}
\aliascntresetthe{conjecture}

\newaliascnt{proposition}{theorem}
\newtheorem{proposition}[proposition]{Proposition}
\aliascntresetthe{proposition}

\theoremstyle{definition}
\newaliascnt{definition}{theorem}
\newtheorem{definition}[definition]{Definition}
\aliascntresetthe{definition}

\newaliascnt{notation}{theorem}
\newtheorem{notation}[notation]{Notation}
\aliascntresetthe{notation}

\newaliascnt{example}{theorem}

\aliascntresetthe{example}

\newaliascnt{examples}{theorem}

\aliascntresetthe{examples}

\newaliascnt{remark}{theorem}
\newtheorem{remark}[remark]{Remark}
\aliascntresetthe{remark}

\newaliascnt{question}{theorem}

\aliascntresetthe{question}

\newaliascnt{questions}{theorem}

\aliascntresetthe{questions}

\newaliascnt{problem}{theorem}

\aliascntresetthe{problem}

\newaliascnt{construction}{theorem}

\aliascntresetthe{construction}

\newaliascnt{setting}{theorem}

\aliascntresetthe{setting}

\newaliascnt{algorithm}{theorem}

\aliascntresetthe{algorithm}

\newaliascnt{observation}{theorem}

\aliascntresetthe{observation}

\newaliascnt{defprop}{theorem}

\aliascntresetthe{defprop}

\DeclareFontFamily{OT1}{pzc}{}
\DeclareFontShape{OT1}{pzc}{m}{it}{<-> s * [1.100] pzcmi7t}{}
\DeclareMathAlphabet{\mathchanc}{OT1}{pzc}{m}{it}

\def\equationautorefname~#1\null{(#1)\null}
\def\sectionautorefname~#1\null{Section #1\null}
\def\subsectionautorefname~#1\null{\S #1\null}

\begin{document}

\baselineskip=16pt

\title[Restrictions on the Betti tables of Licci Ideals]{ 
Restrictions on the Betti tables of Licci Ideals}
\date\today

\thanks{This material is based upon work supported by the National Science Foundation under Grant No. DMS-1928930 and by the Alfred P. Sloan Foundation under grant G-2021-16778, while the authors were in residence at the Simons Laufer Mathematical Sciences Institute (formerly MSRI) in Berkeley, California, during the Spring 2024 semester.
}


\thanks{The second and third authors were partially supported by NSF grants DMS-2502707 and DMS-2502706, respectively.}

\author{Craig Huneke}
\address{Craig Huneke, Department of Mathematics, University of Virginia, Charlottesville, VA 22904} \email{huneke@uva.edu}

\author{Claudia Polini}
\address{Claudia Polini, Department of Mathematics, University of Notre Dame, Notre Dame, Indiana 46556} \email{cpolini@nd.edu}

\author{Bernd Ulrich}
\address{Bernd Ulrich, Department of Mathematics, Purdue University, West Lafayette, Indiana 47907} \email{bulrich@purdue.edu}
\begin{abstract} We introduce several conjectures which mainly deal with restrictions on the Betti tables of licci ideals. We focus on a series of questions that compare the number of generators of homogeneous licci ideals in polynomial rings to the maximal last shift in their graded free resolution. We prove these conjectures
in a large number of cases.
\end{abstract}
 \maketitle

\section{Introduction} 

Linkage (or liaison) is a technique for classifying and investigating ideals and subvarieties, and it has attracted considerable recent attention from a variety of perspectives \cites{KTY2013, Chong, RT19, RTY2020, ERT2020, SW21, RT2023, CVW2019, GNW2024, GNW2412, GNW25, CGNW25, CGNW2503, JRS, MMM, HPU26-3, HPU26-2, HPU26-4}. Let 
$R$ be a Gorenstein ring. Two proper ideals 
$I$ and 
$J$ of
$R$ are called {\it linked} if there is a complete intersection ideal $\mathfrak{a}$
for which
\[
J=\mathfrak{a}:I \qquad \text{and} \qquad I=\mathfrak{a}:J .
\]
Repeating this process generates the {\it linkage class} of an unmixed ideal 
$I$, namely the set of all ideals reachable from 
$I$ by finitely many successive links. The ideals lying in the linkage class of a complete intersection are especially important, and are known as {\it licci} ideals. Since the Cohen-Macaulay property is invariant within a linkage class \cite{PS}, every licci ideal is in particular Cohen-Macaulay.

The interesting paper \cite{MSS} proved
among many other results that if $I$ is a square-free monomial ideal of codimension two in a polynomial ring $S$ such that
$S/I$ is Cohen-Macaulay, then the ideal $I$ is {\it of linear type}, i.e. the Rees algebra and the symmetric algebra of $I$
are naturally isomorphic, and moreover these algebras are Cohen-Macaulay.
Because such ideals are licci, and licci ideals are strongly Cohen-Macaulay \cite{H82}*{Theorem 1.14}, this result is actually equivalent to saying that $I$ satisfies the Artin-Nagata condition $G_{\infty},$ which simply means that
$$\mu(I_P)\leq \dim(S_P)$$ for all primes $P$ containing $I$, where $\mu(-)$ denotes the minimal number of generators (\cite{HSV82}*{Theorem 12.9}). Since
$I$ is a square-free monomial ideal, one only needs to check this
inequality at all primes which are generated by subsets of the variables,
and then one can simply use so-called monomial localization by setting
the variables outside the prime $P$ to 1.  The result is still another
licci square-free monomial ideal and the variables in $P$ become generators of the maximal ideal of the new ring. Thus, the result is true if and only if, for a square-free licci monomial ideal $I$, $$\mu(I)\leq \dim(S).$$

This fact in codimension two leads to our first conjecture:

\begin{conjecture}\label{monconj} Let $S$ be a polynomial ring over
a field, and let $I$ be a licci square-free monomial ideal. Then $\mu(I)\leq \dim(S)$. Equivalently, the ideal $I$ is of linear type and its Rees algebra is Cohen-Macaulay.\end{conjecture}

After computing many examples and proving some cases of
this conjecture, which are included below, it appears that the right conjecture comes from consideration of the free resolution of the
ideal, and is much more general than just the case of square-free monomial ideals.
To state the conjecture, let $I$ be a codimension $c$ homogeneous
licci ideal in a polynomial ring $S$ over a field.  We write the minimal homogeneous $S$-free resolution of $S/I$ in the form 

\begin{equation}\label{introresolution}
0\to \bigoplus_{j\in\mathbb{Z}}S(-j)^{b_{c,j}}  \to \cdots \to \bigoplus_{j\in\mathbb{Z}}S(-j)^{b_{1,j}}\to S\to S/I\to 0 \end{equation}
(since $S/I$ is Cohen-Macaulay, the
length of the resolution is $c$).
We define $T_i(S/I) := \max\{j | b_{i,j}(S/I) \ne 0\}$, and
$t_i(S/I) := \min\{j | b_{i,j}(S/I) \ne 0\}.$ 

\vspace{.2cm}

Our second conjecture is:

\begin{conjecture}\label{regconj} Let $I$ be a licci homogeneous
ideal of codimension $c$ in a polynomial ring $S$ over a field.
Then $$\mu(I)\leq T_c(S/I).$$\end{conjecture}

This second conjecture is more general than the first conjecture. The
reason is that Taylor's resolution shows that for a square-free
monomial ideal, the last twists in the resolution are bounded above
by the number of variables.  Hence for such ideals, $T_c\leq \dim(S)$.
It is then immediate that Conjecture \ref{regconj} implies Conjecture \ref{monconj}.

An equivalent form of this conjecture states that the deviation of $I$, namely $d(I):= \mu(I)-c,$ is bounded above by $T_c(S/I) -c = \reg(S/I)$. After cutting down by linear forms which form a regular sequence on $S/I$, we can reach the case in which the ideal $I$ is
$\m$-primary, where $\m$ is the homogeneous maximal ideal. This reduction does not change the resolution degrees or the number of generators of $I$. The regularity of $S/I$ is then the maximal
socle degree, and Conjecture \ref{regconj} can be rephrased to
say that $\m^{d(I)}\not\subset I$. This statement does not depend upon
the grading, and leads to our final conjecture:

\begin{conjecture}\label{localconj} Let $(S,\m)$ be a regular local ring, and let $I$ be an $\m$-primary licci ideal. Then
$$\m^{d(I)}\not\subset I.$$ \end{conjecture}

The discussion above shows that Conjecture \ref{localconj} implies
Conjecture \ref{regconj}. 

In this paper we will prove all of these conjectures (where they make sense) in many cases, including codimension two Cohen-Macaulay ideals (Theorem \ref{codimtwo}),
codimension three Gorenstein ideals (Theorem \ref{codimthree}), licci ideals with nearly pure resolutions (Theorem \ref{nearlypure}), licci Gorenstein monomial ideals that are equigenerated (Corollary \ref{Monomial}), and licci ideals containing a maximal regular sequence of quadrics (Theorem \ref{quadrics}).  Another section shows that the veracity of the conjecture in the Gorenstein case
already provides strong evidence for the general case.  We use a variety of techniques, including Boij-S\"oderberg theory, facts about Golod rings, the regularity of Tor, basic facts about the resolution of the link of a given ideal, and a numerical restriction on the regularity of homogeneous licci ideals found in \cite{HU87}. In Theorems \ref{generalcase}, \ref{nearlypure}, \ref{Gorenstein},  \ref{Monomial} and Corollary \ref{InitialMonomial} we actually prove the inequality $$\mu(I)\le t_1(S/I)(c-1)+1.$$ This bound is stronger than the one in
Conjecture \ref{regconj} because $t_1(S/I)(c-1)+1 \leq T_c(S/I)$ according to
\cite{HU87}. The bound was first suggested in \cite{H84}, where Huneke asked whether every equigenerated Gorenstein licci ideal satisfies it. It fails in general however, as shown by an example of Terai (see Example 13 in Table 1).

In a previous paper, we have also shown Conjecture \ref{localconj} is true if $I$
is either an $\m$-primary monomial ideal or contains the cube of the maximal ideal \cite{HPU26-3}. In addition the papers
\cites{KTY2013, ERT2020, RT2023, RTY2020, RT19} give a classification of licci ideals for various square-free monomials ideals and some binomial edge ideals. We verified that all these ideals satisfy our conjectures, and using AI prepared a table with the relevant data, rechecking the conclusions.
For the convenience of the reader we provide these tables at the end of this paper giving the examples.

Although this paper concentrates on these specific conjectures, we feel the conjectures reflect that there are deeper and as yet undiscovered
restrictions on the Betti tables of licci homogeneous ideals. There seems to be an interesting interplay between the number and degrees of the generators of such ideals and the last total Betti number (the type) and the sizes of the last twists.
There may well be further dualities between the Betti numbers and the twists appearing in the resolutions of licci ideals.
This paper studies the main conjectures not only to prove them in as many cases as possible, but also as a way to find
further interesting restrictions on these Betti tables.

\vspace{.2cm}

\section{Low Codimension}

In this section we prove the codimension 2 and codimension 3 Gorenstein cases of Conjecture \ref{localconj}.  

\begin{theorem}\label{codimtwo}  Let $(S,\m)$ be a regular local ring of dimension two, and let $I$ be an $\m$-primary ideal of deviation $d = d(I) = \mu(I)-2$. Then $\m^{d}\not\subset I$.
\end{theorem}

\begin{proof}  Set $n+1 = \mu(I)$, so that $d = n-1$.  The Hilbert-Burch theorem implies that $I$ is generated by the $n$ by $n$ minors of a minimal presenting matrix, and in particular $I\subset \m^n$. Therefore, $\m^{n-1}\not\subset I$, as
required.  \end{proof}

\begin{theorem}\label{codimthree}  Let $(S,\m)$ be a regular local ring of dimension three containing a field, and let $I$ be an $\m$-primary ideal  of deviation $d = d(I) = \mu(I)-3$ such that $S/I$ is Gorenstein. Then $\m^{d}\not\subset I$.
\end{theorem}

\begin{proof} 

We use the structure theorem of Buchsbaum and Eisenbud \cite{BE73}.   It implies that if $\mu(I) = 2n+1$, and thus $d = 2n-2$, then $I\subset \m^n$. Suppose that the theorem is false.  Then $\m^{2n-2}\subset I\subset \m^n$. 
This cannot happen if $n\leq 2$ since $I$ is neither the unit ideal nor $\m^2.$
Hence we may assume that $n>2.$ By \cite{Lof86}*{Corollary 2.5} (see also \cite{RS}*{Proposition 6.3})
the ring $S/I$ is Golod. The only Golod Gorenstein rings $S/I$ with $I\subset \m^2$ are hypersurface rings \cite{A96}, contradicting the fact that $I$ has codimension $3$.  
\end{proof}

\section{Nearly pure ideals and Gorenstein equigenerated ideals}

Let $S$ be a polynomial ring over a field $k$, and let $M$ be a Cohen-Macaulay graded $S$-module  of codimension $c$. 
We write the minimal graded resolution of $M$ in the form

\begin{equation}\label{resolution}
0\to \bigoplus_{j\in\mathbb{Z}}S(-j)^{b_{c,j}}  \to \cdots \to \bigoplus_{j\in\mathbb{Z}}S(-j)^{b_{1,j}}\to \bigoplus_{j\in\mathbb{Z}}S(-j)^{b_{0,j}}\to M\to 0. \end{equation}
The numbers $b_{i,j}=b_{i,j}(M)$ are called the graded Betti numbers of $M$.  One can prove the characterization
\[b_{i,j}=\dim_k\Tor_i^S(M,k)_j.\] 
The graded Betti numbers of any module are recorded into the Betti table, where the entry in row $i$ and column $j$ is $b_{i,i+j}$:
\[\begin{array}{c|cccc}
 & 0 & 1 & 2 & \cdots \\ \hline
0 & b_{0,0} & b_{1,1} & b_{2,2} & \cdots \\
1 & b_{0,1} & b_{1,2} & b_{2,3} & \cdots \\
\vdots & \vdots & \vdots & \vdots & \ddots \\
\end{array}
\]

\vspace{.2cm}

For
$i = 0,...,c$ we define $t_i(M) := \min\{j \, |\  b_{i,j}(M) \ne 0\}$ and $T_i(M) := \max\{j \, |\  b_{i,j}(M) \ne 0\}.$ We say that $M$ is {\it equigenerated} if all the minimal generators have the same degree, i.e., that there is only one $j$ such that $b_{0,j}\ne 0$. 

\smallskip

Boij--S\"oderberg theory gives a complete description of the
\emph{rational cone} generated by Betti tables of graded
Cohen--Macaulay modules of fixed codimension $c$.
The main idea is that every Betti table decomposes uniquely as a
positive rational combination of certain extremal Betti tables
coming from modules with \emph{pure resolutions}.

A sequence $\mathbf{d} = (d_0,...,d_c) \in \mathbb{Z}^{c+1}$ is a degree sequence if $d_i > d_{i-1}$ for all $i > 0$.  To each degree sequence we associate a pure Betti diagram 
$\pi(\mathbf{d})$ that has $b_{0,d_0}(\mathbf{d}) = 1$, and entries for $i\geq 1$
\[b_{i,d_i}(\mathbf{d}) =\prod_{1\leq k\ne i}\frac{d_k}{|(d_k-d_i)|}.\] We refer to this array as a {\it diagram} rather than a {\it table} because in general it will not correspond to the actual Betti table of a Cohen-Macaulay module of codimension $c$.

Given two degree sequences $\mathbf{d}$ and $\mathbf{e}$ in $\mathbb Z^{c+1}$, we say that $\mathbf{d} \geq \bf{e}$ if $d_i \geq e_i$ for $i = 0,...,c$. Each degree sequence $\mathbf{d}$ defines a ray in the cone of Betti diagrams. The diagram  $\pi(\mathbf{d})$ is the unique point on this ray with $b_{0,d_0}(\pi(\mathbf{d}))= 1$. Typically,  $\pi(\mathbf{d})$ will have non-integral entries. Let $I$ be a homogeneous ideal of codimension $c$ such that $S/I$ is Cohen-Macaulay.  We set $\mathbf{d^+} := (0,T_1, . . . , T_c)$ and $\mathbf{d^-} := (0,t_1, . . . , t_c)$.  Boij-S\"oderberg theory \cite{ES09} gives that the Betti diagram of  $S/I$ can be expressed as a positive rational sum of strictly increasing pure diagrams bounded below by  
$\mathbf{d^-}$ and bounded above by $\mathbf{d^+}$.  In particular if $t_i = T_i$, then every $\mathbf{d}$ appearing in the Boij-S\"oderberg decomposition of the Betti table of $S/I$ must have
$i$th component equal to $t_i=T_i$.

Our first theorem shows that a certain technical assumption on the relationship
of $t_i$ to $T_1$ implies the conclusion we seek, that $T_c$ bounds the number of
generators of the ideal $I$.

\begin{theorem}\label{generalcase}  Let $S$ be a polynomial ring over a field, and let $I$ be an equigenerated licci ideal of $S$ of codimension $c$. Let the free resolution of $M=S/I$ be given as in $(\ref{resolution}).$ We assume that $t_k > (k-1)T_1$ for all $2\leq k\leq c$.  Then $b_1:= \sum_{j} b_{1,j}\leq (c-1)t_1+1\leq T_c$. In particular, Conjecture $\ref{regconj}$ holds in this case.
\end{theorem}
\begin{proof} We first note that the last inequality follows from the fact that $I$ is licci by \cite{HU87}*{Corollary 5.13}. Decompose the Betti table $B$ of $S/I$ via its Boij-S\"oderberg decomposition \cite{ES09}:

\[B = \sum_{i=1}^k \alpha_i\pi(\bf d_i)\]
where ${\bf d_i} = (0,d_{i,1},...,d_{i,c})$, and $\alpha_i$ are all positive rational numbers. 

There are several important points about this decomposition which we need to use.  First of all, we may assume that the
sequence of pure diagrams is strictly increasing, i.e., ${\bf d^{-}}\leq {\bf d_1} < {\bf d_2} < \cdots < {\bf d_k}\leq {\bf d^{+}}.$ Because $S/I$ is cyclic, we must have that the  $\sum_{i=1}^k \alpha_i = 1$, since the zeroth Betti number in the resolution is $1$, and $b_{0,d_{i,0}}({\bf d_i}) = 1$.

Recall that the Betti table of $\pi(\bf d_i)$ has entries,

\[b_{j,d_{i,j}}({\bf d_i}) =\prod_{1\leq k\ne j}\frac{d_{i,k}}{|(d_{i,k}-d_{i,j})|}.\]
We claim that it suffices to prove that for all $1\leq i\leq k$  $$b_{1,d_{i,1}}({\bf d_i})\leq 
d_1(c-1)+1.$$
In that case, 
$$b_1(S/I) = \sum_{i = 1}^k \alpha_i b_{1,d_{i,1}}({\bf d_i})\leq \sum_{i = 1}^k \alpha_i (d_1(c-1)+1) = d_1(c-1)+1,$$
where the last equality holds since $\sum_{i=1}^k \alpha_i = 1$. This is the statement of the theorem.

Now fix one of the pure diagrams $\mathbf{d_i}:= \mathbf{d}$, and for simplicity relabel the degrees for this sequence as $(0,d_1,d_2,d_3,...,d_c)$.   Although this degree sequence may not correspond to an actual module, we know that $\mathbf{d^{-}}\leq \mathbf{d}\leq \mathbf{d^+}$; in particular,
$T_k\geq d_k\geq t_k$ for all $k$. We need to prove that 

\[ \prod_{2\leq k\leq c}\frac{d_{k}}{(d_{k}-d_{1})} = b_{1,d_{1}}({\bf d}) \leq d_1(c-1)+1.\]

The function $f(x) = \frac{x}{(x-d_1)} = 1 + \frac{d_1}{(x-d_1)}$ is  a decreasing function. In particular, 
$$\frac{d_{k}}{(d_{k}-d_{1})}\leq \frac{(k-1)d_1+1}{(k-2)d_{1}+1}$$ 
since $d_k \geq t_k$, and by assumption $t_k > (k-1)T_1\geq (k-1)d_1$ for all $2\leq k\leq c$. Hence
\[b_{1,d_{1}}(\mathbf{d})\leq \frac{((c-1)d_1+1)}{((c-2)d_1+1)}\cdot\frac{((c-2)d_1+1)}{((c-3)d_1+1)}\cdots\frac{(2d_1+1)}{(d_1+1)}\cdot\frac{(d_1+1)}{1} = (c-1)d_1+1.\] 
\end{proof}

We apply this theorem in two cases: when $I$ is a Gorenstein ideal which is equigenerated, and when $S/I$ has a nearly pure resolution.
We say that $S/I$ has a {\it nearly pure} resolution if it has a minimal graded resolution of the form
\begin{equation}\label{pure}
0\to S^{b_c}(-d_c)  \to \cdots \to S^{b_3}(-d_3)\to \oplus_{i =1}^{b_2}S(-d_{2,i}) \to S^{b_1}(-d_1)\to S\to S/I\to 0. 
\end{equation}
In other words, we assume all the shifts in the minimal resolution in every homological degree have only one internal degree, except possibly for the first syzygies of the ideal $I$. We first treat the nearly pure case.

\begin{theorem}\label{nearlypure}  Let $S$ be a polynomial ring over a field, and let $I$ be a licci ideal of $S$ of codimension $c\geq 3$, having a nearly pure resolution as in $(\ref{pure})$.  Then $b_1\leq (c-1)d_1 + 1$. In particular $b_1\leq d_c$.
\end{theorem}

\begin{proof} Since $T_c = d_c$ and $T_1 = t_1 = d_1$, the theorem follows immediately from Theorem \ref{generalcase} provided we prove that
$$t_k > (k-1)T_1$$ for all $2\leq k\leq c.$

We use two fundamental inequalities concerning these integers.  First, from \cite{HU87}*{Corollary 5.13}, the fact that $I$ is licci implies that
$$t_c = d_c= T_c > (c-1)t_1 = (c-1)T_1,$$ which gives the required inequality for $k = c$.
Secondly, using \cite{EHU06}*{Corollary 4.2} it follows that \begin{equation}\label{EHU} T_c\leq T_k + (c-k)T_1\end{equation} for every $2\leq k\leq c$. Our assumptions imply that $T_k = t_k = d_k$ for every $k\ne 2$, so we may combine the two inequalities to obtain that 
$$t_k = T_k > (k-1)T_1$$ for all $3\leq k\leq c$.

This inequality is satisfied as well for $k = 2$, since $t_2 > t_1 = T_1$.\end{proof}

Another case in which we can verify the hypothesis of Theorem \ref{generalcase} is
given in the following result:



\begin{theorem}\label{Gorenstein} Let $I$ be an equigenerated  licci ideal
of codimension $c$ in a polynomial ring $S$ over a field such that $S/I$ is Gorenstein.  Suppose for all $1\leq p\leq c-1$, $T_p\leq p\cdot T_1$, where we adopt the notation above. Then
$b_1\leq (c-1)t_1+1$. In particular, $b_1\leq T_c$.  \end{theorem}

\begin{proof} We use the same notation as in (\ref{resolution}) with the same
definitions of $T_i$ and $t_i$. It suffices to prove that $t_k > (k-1)T_1$ for all $2\leq k\leq c$;
then Theorem \ref{generalcase} gives the conclusion. Since $S/I$ is
Gorenstein, we know that  $t_c = T_c$,
and $T_c > (c-1)T_1$ since $I$ is licci. Using again that $S/I$ is Gorenstein, 
$$T_{c-k} = T_c - t_{k}$$ for all $0\leq k\leq c.$ Fix an integer $k$, $2\leq k\leq c-1$.  For $1\leq c-k\leq c-1$, we are assuming that $T_{c-k}\leq (c-k)T_1$.  As 
$T_{c-k} = T_c-t_{k}$, the assumed inequality becomes $T_c-t_{k}\leq (c-k)T_1$,
for all $0\leq k\leq c-1$. 
Hence $T_c-(c-k)T_1\leq t_k$. Since $I$ is licci, \cite{HU87}*{Corollary 5.13} proves that $T_c > (c-1)T_1.$ Combining these two inequalities gives 
$(k-1)T_1 < t_k$ for all $0\leq k\leq c-1$. For $k = c$ we know that $t_c = T_c$ since
$S/I$ is Gorenstein, and $T_c > (c-1)T_1$ since $I$ is licci. We may apply Theorem \ref{generalcase} to finish the proof. \end{proof}

As an application of this theorem, we obtain the following two corollaries:

\begin{corollary}\label{Monomial} Let $I$ be a monomial licci equigenerated ideal
of codimension $c$ in a polynomial ring $S$ over a field such that $S/I$ is Gorenstein. Then $b_1\leq (c-1)d_1 + 1$. In particular $b_1\leq d_c$.
 \end{corollary}

\begin{proof} By Theorem \ref{Gorenstein} it suffices to prove that $T_p\leq p\cdot T_1$ for all $1\leq p\leq c.$  However, from the Taylor resolution, the twists at the $p^{\rm th}$ step of the multigraded free resolution of $S/I$ all arise from the least common multiple
of $p$ generators of $I$. All generators of $I$ have degree $T_1$ by assumption, so
the degree of the least common multiple of any $p$ of them has degree at most $p\cdot T_1$. \end{proof}

\begin{corollary}\label{InitialMonomial} Let $I$ be a  licci ideal of codimension $c$ generated in one degree $\delta$
in a polynomial ring $S$ over a field such that $S/I$ is Gorenstein. Suppose that there is a term order such that
the initial ideal ${\rm in} (I)$ is generated in degree $\delta$. Then
$b_1\leq (c-1)\delta + 1$. In particular, $b_1\leq T_c$. \end{corollary}

\begin{proof} By Theorem \ref{Gorenstein} it suffices to prove that $T_p\leq p\cdot T_1$ for all $1\leq p\leq c.$ In general \cite{MS}*{Theorem 8.29},
$T_p(S/I)\leq T_p(S/{\rm in}(I)).$ As observed in the proof of Corollary \ref{Monomial}, $T_p(S/{\rm in}(I))\leq p\cdot\delta = p\cdot T_1$. \end{proof}

\begin{remark} Another class of rings $S/I$ which satisfy the condition that $T_p\leq p\cdot T_1$ is the class of quadratic Koszul algebras; indeed, if $S/I$ is Koszul then $T_i(S/I)\leq 2i$ for all $i$ \cites{ACI10, Conca14}.
The licci quadratic Koszul algebras were recently classified
in \cite{MMM}*{Theorem A} and all of them satisfy our conjectures.  However, in Section \ref{regularquadrics} we prove our conjectures for every ideal containing a maximal regular sequence of quadrics; they do not have to be generated by
quadrics. \end{remark}

\section{Sums of Links}

In this section we explore the implications of assuming the validity of the main conjecture for Gorenstein licci ideals.  It turns out that this assumption has strong consequences for the general case.  

\begin{notation}\label{setup} Throughout this section, we let $S$ be a polynomial ring over an infinite field $k$, and let $I$ be a homogeneous ideal of codimension $c$ that is generically a complete intersection, but which is not itself a complete intersection
. We fix a maximal homogeneous regular sequence $f_1,\ldots,f_c$ in $I$ of degrees $a_1,\ldots,a_c$, and in addition assume that this regular sequence generates $I$ after localizing at every minimal prime
of $I$.  We require that $f_1,...,f_c$ are part of a minimal generating set of $I$. We note that $c$ is strictly less than the dimension of $S$; otherwise the assumptions would imply $I$ is already a complete intersection. Let $J = (f_1,...,f_c):I$.  We set $L = I+J$, which is always Gorenstein of codimension $c+1$ (and if $I$ is licci, so is $L$
by \cite{U90}*{Theorem 2.1}). 

If $H$ is a perfect homogeneous ideal, by $T_i(S/H)$ (respectively $t_i(S/H)$) we denote the maximum (respectively minimum) shift
in the minimal graded free resolution of $S/H$ in homological degree $i$. The invariant we wish to study is given in the next definition:

\begin{definition}\label{difference} For every homogeneous ideal $H$
of codimension $c$ in a polynomial ring $S$ such that $H$ is perfect, we define
$$\Delta(H) = T_c(S/H)-\mu(H),$$
the difference between the largest last shift in the minimal free resolution of $S/H$ and the number of generators of $H$. \end{definition}

The main conjecture of this paper, Conjecture \ref{regconj}, can be restated by saying
$I$ licci implies that $\Delta(I)\geq 0.$ 
What we are able to prove in Theorem \ref{technical} is that
there is a very strong relationship connecting 
$\Delta(I), \Delta(J)$ and $\Delta(L)$. We need to compare these
invariants by considering the free resolutions of the ideals.

We let
\begin{equation}\label{Ires}
\mathbf{F}:\ 0 \longrightarrow \bigoplus_{i=1}^{b_c} S(-d_{ci}) \longrightarrow \cdots
\longrightarrow \bigoplus_{i=1}^{b_1} S(-d_{1i}) \longrightarrow S
\longrightarrow S/I \longrightarrow 0.
\end{equation}
be the minimal graded free resolution of $S/I$.
\end{notation}

\begin{remark} The resolution of $S/(f_1,...,f_c)$ is given by the Koszul complex
and has the form:

\begin{equation}\label{Kres}
\mathbf{K}:\ 0 \longrightarrow S(-a_1-\cdots-a_c) \longrightarrow \cdots \longrightarrow \bigoplus_{1\leq i\leq c}S(-a_i) \longrightarrow S
\longrightarrow S/(f_1,\ldots,f_c)\longrightarrow 0.
\end{equation}
We set $A:= a_1+\cdots + a_c$.

\medskip

A homogeneous $S$-resolution of $S/J$ can be obtained in the
following way: 
let $u:\mathbf{K}\to \mathbf{F}$ be a homogeneous morphism of complexes induced by the inclusion of $(f_1,...,f_c)\subset I$:
\[
\begin{tikzcd}[column sep=large, row sep=large]
0 \arrow[r] &
\displaystyle\bigoplus_i S(-d_{ci}) \arrow[r] & \cdots \arrow[r] &
\displaystyle\bigoplus_i S(-d_{1i}) \arrow[r] & S \\
0 \arrow[r] &
S(-A) \arrow[u,"u_c"] \arrow[r] &
\cdots  \arrow[r] & \bigoplus_i S(-a_i) \arrow[u,"u_1"] \arrow[r] &
S\arrow[u,"u_0"] .
\end{tikzcd}
\]
Then the $S$-dual of the mapping cone of $u$ yields a possibly non-minimal homogeneous
resolution of the link $S/J$ \cite{PS}:
\begin{equation}\label{mappingcone}
0 \longrightarrow \bigoplus_i S(-A+d_{1i}) \longrightarrow \cdots
\longrightarrow \bigoplus_i S(-A+d_{ci}) \oplus \bigoplus_i S(-a_i)
\longrightarrow S \longrightarrow S/J \longrightarrow 0.
\end{equation}
\end{remark}

\begin{remark}\label{mingencanonical}
The link $J$ of $I$ is generated by the entries $I_1(u_c)$ of a matrix representation of $u_c$, together with the regular sequence $f_1,...,f_c$.  This is not necessarily a minimal generating set as there can be cancellation of
some of the $f_i$ due to splitting of $u_{c-1}$.  We will make this fact precise below.  However, since $I$ is not a complete intersection, the map $u_c$ cannot have a unit, else $J = S,$ forcing $I$ to be a complete intersection. Moreover, the canonical module of $S/I$ is isomorphic to $J/(f_1,...,f_c) = (I_1(u_c) + (f_1,...,f_c))/(f_1,...,f_c)$. As the canonical module has exactly the number of generators as the rank of the
last free module in the resolution of $S/I$, it follows that the elements in $I_1(u_c)$ are independent generators, even modulo the regular sequence of $f_i$.

\end{remark}

\begin{theorem}\label{technical}  We adopt the notation from Notation \ref{setup}. Then, 
$$\Delta(I)+\Delta(J)-\Delta(L) = T_c(S/I) - t_1(S/I) -c.$$
If  $\Delta(L)\geq 0$, then either $\Delta(I)\geq 0$ or $\Delta(J)\geq 0$.
\end{theorem}

\begin{proof} We  need to compute the various invariants for $I,J$ and $L$.
There is an exact sequence
\begin{equation}\label{ses} 0\to S/(f_1,...,f_c)\to S/I\oplus S/J\to S/L\to 0.\end{equation}

First, consider $L$. We find the last twist of the resolution of $S/L$ by computing
$\Tor_{c+1}(S/L,k)$. The long exact sequence coming from (\ref{ses}) shows that this Tor embeds in $\Tor_c(S/(f_1,...,f_c),k)$. Since both
are one-dimensional over $k$ (both are Gorenstein), this map is an isomorphism, and proves that
the last twist in the resolution of $S/L$ is exactly $A$, i.e. $$T_{c+1}(S/L) = A.$$

Define an integer $s$ as follows: $s = \dim_k(u_{c-1}\otimes_Sk)$. In other words, $s$ is exactly the number of Koszul relations on the regular sequence which
split away from the resolution of $S/I$ at the $(c-1)^{\rm st}$ comparison map.
For the number of generators of $L$ we use the following lemma:

\begin{lemma} Let the notation be as in Notation \ref{setup}. Then $\mu(L) = b_1+b_c-s.$ \end{lemma}

\begin{proof} We use the short exact sequence,
$$0 \longrightarrow I\cap J\longrightarrow I\oplus J\longrightarrow L\longrightarrow 0.$$
After tensoring with the residue field, we see that $\mu(L) = \mu(I) + \mu(J) - p,$ where $p$ is the
dimension of the image of $(I\cap J)/\m(I\cap J)$ in $I/\m I\oplus J/\m J$.  Since $I\cap J = (f_1,...,f_c)$ and the regular sequence $f_1,...,f_c$ is assumed to be part of a minimal generating set of $I$, the dimension of this image is exactly $c$. Consider $\mu(J)$. From the
resolution of $S/J$ as in (\ref{mappingcone}) one sees that $\mu(J) = b_c+c-s.$ This follows as the minimal generators of $J$ come from the entries of a matrix representation of $u_c$, together with elements of the regular sequence in the link which are not canceled by splitting of $u_{c-1}$. (see \ref{mingencanonical}). Hence
$$\mu(L) = \mu(I)+ \mu(J) - c = b_1+ (b_c+c-s) -c = b_1+b_c-s.$$\end{proof}

Finally, from the mapping cone (\ref{mappingcone}) using that the image of $u_c$ lies in the maximal ideal times the free module, one obtains that $$T_c(S/J) = A - t_1(S/I).$$

We summarize what we have shown regarding $\Delta$:
$$\Delta(I) = T_c(S/I)-b_1.$$
$$\Delta(J) = (A - t_1(S/I)) -(b_c+c-s) = A - t_1(S/I) -b_c - c +s.$$
$$\Delta(L) = A - (b_1+b_c-s) = A - b_1 -b_c +s.$$

It follows immediately that $$\Delta(I)+\Delta(J)-\Delta(L) = T_c(S/I)-t_1(S/I)-c,$$ proving the first statement of the theorem.
Now assume that $\Delta(L)\geq 0.$ Note that $t_1(S/I) + c- 1 \leq T_c(S/I)$ since the matrices in the
minimal resolution of $S/I$ have to have at least degree one entries at every
step in the resolution. It follows that
$$\Delta(I)+\Delta(J)\geq -1. $$  
It is now clear that if either of $\Delta(I)$ or $\Delta(J)$ is negative, the other must be nonnegative.   \end{proof}

\begin{remark} In this section we have not yet assumed that
$I$ is licci. If $I$ is licci then so are $J$ and $L$, as mentioned
above. Our conclusion shows that if Conjecture \ref{regconj} holds for
the Gorenstein licci ideal $L$, then it must hold for either $I$ or
$J$.  In some sense the truth of the Gorenstein case of the conjecture implies that at least half of all licci ideals satisfy the conjecture
as well. In fact, when $I$ is licci, we can say more, as in the next theorem.
\end{remark}

We let $d(I) = \mu(I)-c$ be the deviation of $I$ in the statement of the next theorem.

\begin{theorem}\label{extremaldeviation} Let the notation be as in Notation \ref{setup}, and assume that $I$ is licci. In addition
assume that $\Delta(L)\geq 0$, i.e., that Conjecture \ref{regconj} holds for the Gorenstein licci ideal $L$. Then
$$A\geq b_1(S/I) + b_c(S/I) -s.$$ If
$$t_1(S/I)\leq d(I),$$
then $\Delta(J)\geq 0$. If $$d(I)\leq (c-1)(t_1(S/I)-1),$$
then $\Delta(I)\geq 0$. In particular, if 
$$t_1(S/I)\leq d(I) \leq (c-1)(t_1(S/I)-1),$$ then Conjecture \ref{regconj} holds for both $I$ and $J$ as well.
\end{theorem}

\begin{proof}  The first statement is a restatement of what we proved in Theorem \ref{technical}:
$$\Delta(L) = A - (b_1+b_c-s) = A - b_1 -b_c +s.$$ 

To prove the second statement, assume that $t_1(S/I)\leq d(I)$. We calculate that
$$\Delta(J)-\Delta(L) = [A - t_1(S/I) -b_c - c +s] - [A - b_1 -b_c +s] = b_1-t_1(S/I)-c.$$ Thus $$\Delta(J) = \Delta(L) + d(I) - t_1(S/I)\geq d(I)-t_1(S/I).$$ Our assumption then proves that $\Delta(J)\geq 0$.

To prove the third statement, assume that
$$d(I)\leq (c-1)(t_1(S/I)-1).$$ In this case, $b_1 = \mu(I)\leq (c-1)t_1(S/I) + 1\leq T_c(S/I)$, since $I$
is licci. Thus $\Delta(I)\geq 0$.   \end{proof}

\begin{remark} It might seem that the roles of $I$ and $J$ are completely symmetric, but for our
calculations, we have used that the linking regular sequence is part of a minimal generating sequence of $I$.  In general
it will not be part of a minimal generating sequence of $J$ as well. One must be careful in trying to exchange the roles
of the two ideals.
\end{remark}


\section{Equigenerated ideals and ideals containing a regular sequence of quadrics}\label{regularquadrics}

\begin{proposition}\label{ACI}
Let $S=k[x_1,\ldots, x_n]$ be a polynomial ring over a field $k$ and 
$I$ be
an $\m$-primary ideal generated by forms of degree $\delta\ge 3$ that is not a complete intersection.  Write $s$ for the highest socle degree of $S/I.$ If $s\ge (\delta-1)n-1,$ then $d(I)\le n<s.$ Furthermore if $d(I)\ge n-2,$ then $I$ is licci. 
\end{proposition}
\begin{proof}  The ideal $I$ contains a regular sequence $\underline{\alpha}$ of $n$ forms of degree $\delta$ and the socle of $S/(\underline{\alpha})$ sits in degree $ (\delta-1)n.$ 
Since \(I\) is not a complete intersection and \(s\ge (\delta-1)n-1\), it follows that
\(s=(\delta-1)n-1\).

Write $a=\dim_k [{\rm soc} \left(\frac{S}{I}\right)]_{(\delta-1)n-1}.$  From the free resolution in (\ref{mappingcone}) we see that the ideal $J=(\underline{\alpha}):I,$ contains $a\ge 1$ independent linear forms (see \ref{mingencanonical}), and furthermore the last module in the free resolution of $J$ is $S^{d(I)}(-\delta n+\delta).$ Hence $J$ contains a regular sequence $\underline{\beta}$ of $a$ linear forms and $n-a$ forms of degree $\delta.$ 


Let $K=(\underline{\beta}): J.$  We write the end of the mapping cone construction whose dual provides a possibly non-minimal free resolution of $S/K$. Let $\mathbf{K}$ be the Koszul complex resolving $S/(\underline{\beta}),$ let $\mathbf{F}$  be a minimal graded free resolution of $S/J,$ and let $u:\mathbf{K}\to \mathbf{F}$ be a homogeneous morphism of complexes induced by the inclusion of $(\underline{\beta})\subset J$:
\[
\begin{tikzcd}[column sep=large, row sep=large]
0 \arrow[r] & S^{d(I)}(-\delta n +\delta) \arrow[r] & \cdots  \\
0 \arrow[r] &
S(-\delta n+(\delta-1)a) \arrow[u,"u_n"] \arrow[r] &
\cdots   .
\end{tikzcd}
\]
The existence of a non-zero map $u_n$ forces $\delta n -\delta a +a\ge \delta n-\delta,$ thus $a\ge \delta (a-1).$ Thus $a=1$ because $\delta \ge 3.$ In this case we claim that $K$ contains $d(I)$ linearly independent linear forms, and hence $d(I)\le n.$ Indeed, according to (\cite{PS} see also \cite{HU85}*{Proposition 2.2}) 
\[K/(\underline{\beta}) \cong \omega_{S/J}  \cong I/(\alpha)\] 
up to shifts, so these $d(I)$ linear forms are minimal generators of $K,$ and in particular are linearly independent. 

If $d(I)\ge n-2,$ then after a change of variable we have $K=(x_1, \ldots, x_{n-2}, K'),$ where $K'\subset k[x_{n-1},x_n].$ Since $K'$ is licci, $K$ is licci and $I$ is then licci as well. 
\end{proof}

\begin{theorem}\label{quadrics} Let $S=k[x_1,\ldots, x_n]$ be a polynomial ring over a field $k$ and 
$I$ be
a  homogeneous ideal containing a regular sequence of $n$ quadrics. 
Write $s$ for the highest socle degree of $S/I.$ If $I$ is licci, then the deviation of $I$ is at most $s=s(I)$, i.e. Conjecture $\ref{localconj}$ holds for $I$. 
\end{theorem}
\begin{proof} 
We proceed by induction on $n.$ If $n=1$ the claim is clear and if $n=2$ the claim follows from Theorem~\ref{codimtwo}. Therefore we may assume $n\ge 3.$ 

If $I$ contains a linear form $\ell$, after a change of variables we can write  $I=(I', x_n)$ with $I'\subset k[x_1, \ldots, x_{n-1}].$ We claim that the induction hypothesis applies to $I'.$ Indeed, $I'$ is still homogeneous, contains a regular sequence of $n-1$ quadrics, and 
$I'$ is licci according to \cite{HU07}*{Proposition 3}. Notice that $d(I')=d(I)$ and $s$ is still the top socle degree of $I'$. By induction we obtain $d(I)=d(I')\le s,$ as claimed. 

Thus we may assume that $I \subset \m^2.$ Since 
$I$ is licci, the maximal last shift in a minimal homogeneous free resolution of $I$ is at least $2(n-1)+1=2n-1$ \cite{HU87}*{Corollary 5.13(a)}, which forces $s\ge n-1.$ On the other hand the socle of a complete intersection defined by $n$ quadrics has degree $n,$ thus either $I$ is a complete intersection and $d(I)=0$ or $s=n-1.$
Hence we may assume $s=n-1.$ 

Write $a$ for the number of socle elements in $S/I$ of degree $n-1.$ Notice $a\ge 1.$ Choose a regular sequence $\underline{\alpha}$ of $n$ quadrics contained in $I.$   The ideal $J=(\underline{\alpha}) :I$ is licci, has type $d(I),$  and contains exactly $a$ linearly independent linear forms, as can be seen from the free resolution given in (\ref{mappingcone}). Without loss of generality, we may write $J=(J', x_{n-a+1}, \ldots, x_n),$ where $J'\subset k[x_1, \ldots, x_{n-a}].$ The ideal $J'$ is still licci \cite{HU07}*{Proposition 3}, is contained in $(x_1, \ldots,x_{n-a})^2,$ and contains a regular sequence of $n-a$ quadrics. Thus either $J'$ is generated by the regular sequence of quadrics, in which case $d(I)\le 1\le n-1=s,$ or  $s(J')=n-a-1$ and hence $s(J)=n-a-1.$ Thus we may assume $s(J)=n-a-1.$

Let $b$ be the number of socle elements in $S/J$ of degree $n-a-1.$  Choose a regular sequence $\underline{\beta}$ of $n-a$ quadrics in $J'$, so that $\underline{\gamma}=\underline{\beta}, x_{n-a+1}, \ldots, x_n$ form a regular sequence. 
The ideal $K=(\underline{\gamma}):J $ is licci and contains exactly $a+b$ linearly independent linear forms. 
We may write $K=(K',  x_{n-a-b+1}, \ldots, x_n),$ where  $K'\subset k[x_1, \ldots, x_{n-a-b}].$ The ideal 
$K'$ is still licci, is contained in $(x_1, \ldots,x_{n-a-b})^2,$ and contains a regular sequence 
of $n-a-b$ quadrics. Thus $s(K)=s(K')\le n-a-b$. By induction $d(K)=d(K')\le s(K')=s(K)\le n-a-b.$

Notice that $$I/(\underline{\alpha})\cong\omega_{S/J}\cong K/(\underline{\gamma})$$
up to shifts, (7\cite{PS} see also \cite{HU85}*{Proposition 2.2}).
Since $\underline{\alpha}$ are part of a minimal generating sequence of $I$, it follows that $$d(I)=\mu(I/(\underline{\alpha}))=\mu(K/(\underline{\gamma})).$$

We claim that $\mu(K/(\underline{\gamma}))\le\mu(K)-n+b=d(K)+b,$ in other words, the image 
of $(\gamma)$ in $K/\m K$ has dimension at least $n-b$ over $k.$ Let $^{-}$ denote images under
the natural projection map $S \rightarrow S/(x_{n-a-b+1}, \ldots, x_n).$ It suffices to show that the image of $(\gamma)$ in $\overline{K}/\overline{\m}\overline{K}$ has dimension at least $n-a-b$ over $k.$ This follows from
Krull's Altitude Theorem because $\overline K$ contains no linear forms, $(\overline{\underline{\gamma}})$ is generated by quadrics, and $(\overline{\underline{\gamma}})$
has codimension at least $n-a-b.$

In particular
$$
d(I)=\mu(K/(\underline{\gamma}))\le d(K)+b \le s(K)+b\le n-a \le n-1=s
$$
as asserted. 
\end{proof}



\vspace{.3cm}

\section{The conjecture for known classes of licci ideals} 

The following tables list ideals from \cites{RT19,KTY2013,RTY2020,ERT2020,RT2023}. Each table includes the deviation and regularity, so that the conjecture can be checked for every ideal listed. A few of these ideals are not licci, and these in fact fail to satisfy the conjecture, however every licci ideal satisfies the conjecture.

\begin{longtable}{@{}l c c c c c@{}}
\caption{\cite{RT19}*{Theorem 3.5} deviation and maximal socle degree .}
\label{tab:RT2019-thm35}\\
\toprule
\multicolumn{1}{c}{$I$} & $n$ & $c$ & $\mu(I)$ & $d(I)$ & socle deg.\\
\midrule
\endfirsthead
\toprule
\multicolumn{1}{c}{$I$} & $n$ & $c$ & $\mu(I)$ & $d(I)$ & socle deg.\\
\midrule
\endhead
\bottomrule
\endfoot

$(x_1x_2x_3x_4x_5)$ & 5 & 1 & 1 & 0 & 4\\
\midrule
$(x_1x_2,\,x_3x_4x_5x_6)$ & 6 & 2 & 2 & 0 & 4\\
$(x_1x_2x_3,\,x_4x_5x_6)$ & 6 & 2 & 2 & 0 & 4\\
\midrule
$(x_1x_2,\,x_3x_4,\,x_5x_6x_7)$ & 7 & 3 &  3 & 0 & 4\\
$(x_3x_6,\,x_3x_7,\,x_4x_7,\,x_1x_2x_4x_5,\,x_1x_2x_5x_6)$ & 7 & 3 & 5 & 2 & 4\\
$(x_3x_6,\,x_3x_7,\,x_1x_4x_7,\,x_2x_4x_6,\,x_1x_2x_4x_5)$ & 7 & 3 & 5 & 2 & 4\\
$(x_4x_7,\,x_1x_3x_6,\,x_1x_3x_7,\,x_2x_4x_5,\,x_2x_4x_6)$ & 7 & 3 & 5 & 2 & 4\\
$(x_1x_3x_6,\,x_1x_3x_7,\,x_1x_4x_7,\,x_2x_4x_5,\,x_2x_4x_7,\,x_2x_5x_6,\,x_3x_5x_6)$ & 7 & 3 & 7 & 4 & 4\\
\midrule
$(x_1x_2,\,x_3x_4,\,x_5x_6,\,x_7x_8)$ & 8 & 4 & 4 & 0 & 4\\
$(x_3x_6,\,x_3x_7,\,x_4x_5,\,x_4x_7,\,x_5x_6,\,x_1x_2x_8)$ & 8 & 4 & 6 & 2 & 4\\
$(x_2x_8,\,x_3x_6,\,x_3x_7,\,x_4x_7,\,x_1x_4x_5,\,x_1x_5x_6)$ & 8 & 4 & 6 & 2 & 4\\
$(x_2x_8,\,x_4x_7,\,x_5x_6,\,x_1x_3x_6,\,x_1x_3x_7,\,x_1x_3x_8,\,x_2x_4x_5)$ & 8 & 4 & 7 & 3 & 4\\
$(x_2x_8,\,x_3x_7,\,x_4x_7,\,x_1x_3x_6,\,x_1x_3x_8,\,x_1x_5x_6,\,x_2x_4x_5,\,x_4x_5x_6)$ & 8 & 4 & 8 & 4 & 4\\
$(x_2x_8,\,x_3x_7,\,x_3x_8,\,x_4x_7,\,x_1x_3x_6,\,x_1x_5x_6,\,x_2x_4x_5)$ & 8 & 4 & 7 & 3 & 4\\
$(x_2x_8,\,x_3x_6,\,x_3x_7,\,x_3x_8,\,x_4x_7,\,x_1x_5x_6,\,x_1x_2x_4x_5)$ & 8 & 4 & 7 & 3 & 4\\
\midrule
$(x_1x_9,\,x_2x_8,\,x_3x_6,\,x_3x_7,\,x_4x_5,\,x_4x_7,\,x_5x_6)$ & 9 & 5 & 7 & 2 & 4\\
$(x_1x_9,\,x_2x_8,\,x_3x_6,\,x_3x_7,\,x_3x_8,\,x_4x_7,\,x_5x_6,\,x_2x_4x_5)$ & 9 & 5 & 8 & 3 & 4\\
$(x_1x_9,\,x_2x_8,\,x_2x_9,\,x_3x_8,\,x_4x_7,\,x_5x_6,\,x_1x_3x_6,\,x_1x_3x_7,\,x_2x_4x_5)$ & 9 & 5 & 9 & 4 & 4\\
$(x_1x_9,\,x_2x_8,\,x_3x_6,\,x_3x_7,\,x_3x_8,\,x_3x_9,\,x_4x_7,\,x_5x_6,\,x_1x_2x_4x_5)$ & 9 & 5 & 9 & 4 & 4\\
\end{longtable}

\begin{longtable}{@{}l c c c c c@{}}
\caption{
\cite{RT2023}*{Theorem 4.1}: deviation and maximal socle degree.}
\label{tab:RT2023-thm41}\\
\toprule
\multicolumn{1}{c}{$I$} & $f_2$ & $r$ & $\mu(I)$ & $d(I)$ & socle deg.\\
\midrule
\endfirsthead
\toprule
\multicolumn{1}{c}{$I$} & $f_2$ & $r$ & $\mu(I)$ & $d(I)$ & socle deg.\\
\midrule
\endhead
\bottomrule
\endfoot

$(cd,\,be,\,af)$ & 8 & 1 & 3 & 0 & 3\\
$(acd,\,be,\,ade,\,bf,\,cf)$ & 8 & 1 & 5 & 2 & 3\\
\midrule
$(df,\,cd,\,cf,\,abe)$ & 9 & 2 & 4 & 1 & 3\\
$(be,\,bf,\,cf,\,acd,\,adef)$ & 9 & 2 & 5 & 2 & 3\\
$(be,\,bf,\,cde,\,acd,\,adf,\,acef)$ & 10 & 2 & 6 & 3 & 3\\
$(bf,\,cd,\,aef,\,abe)$ & 10 & 2 & 4 & 1 & 3\\
$(cf,\,df,\,abe,\,acd,\,bde)$ & 10 & 2 & 5 & 2 & 3\\
$(af,\,be,\,bcd,\,cdf,\,acde)$ & 10 & 2 & 5 & 2 & 3\\
$(cd,\,abe,\,abf,\,aef,\,bdf,\,cef)$ & 11 & 2 & 6 & 3 & 3\\
$(bf,\,abe,\,acd,\,ace,\,bcd,\,def)$ & 11 & 2 & 6 & 3 & 3\\
\midrule
$(cf,\,df,\,abe,\,cde,\,abcd)$ & 11 & 3 & 5 & 2 & 3\\
$(bf,\,cf,\,ade,\,bde,\,abcd,\,abce)$ & 11 & 3 & 6 & 3 & 3\\
$(bf,\,cd,\,ade,\,abce,\,acef)$ & 11 & 3 & 5 & 2 & 3\\
$(be,\,acd,\,adf,\,bcf,\,cef)$ & 12 & 3 & 5 & 2 & 3\\
$(acd,\,abce,\,ade,\,bde,\,abf,\,cf)$ & 12 & 3 & 6 & 3 & 3\\
$(bcd,\,abce,\,bde,\,cde,\,af,\,bef)$ & 12 & 3 & 6 & 3 & 3\\
$(bcd,\,bce,\,bde,\,cde,\,af)$ & 12 & 3 & 5 & 2 & 3\\
\midrule
$(abcd,\,abe,\,acde,\,bcde,\,cf,\,df)$ & 12 & 4 & 6 & 3 & 3\\
$(acd,\,abce,\,ade,\,cde,\,bf,\,acef)$ & 13 & 4 & 6 & 3 & 3\\
$(abcd,\,abce,\,ade,\,bde,\,abf,\,cf)$ & 13 & 4 & 6 & 3 & 3\\
\end{longtable}

\begin{longtable}{@{}l c c c c c@{}}
\caption{
\cite{RTY2020}*{Theorem 3.2}: deviation, maximal socle degree, and licci status.}
\label{tab:RTY2020-thm32}\\
\toprule
\multicolumn{1}{c}{$I$} & $f_2$ & $\mu(I)$ & $d(I)$ & socle deg. & licci\\
\midrule
\endfirsthead
\toprule
\multicolumn{1}{c}{$I$} & $f_2$ & $\mu(I)$ & $d(I)$ & socle deg. & licci\\
\midrule
\endhead
\bottomrule
\endfoot

$(df,\,cd,\,cf,\,abe)$ & 9 & 4 & 1 & 3 & yes\\
$(be,\,bf,\,cf,\,acd,\,adef)$ & 9 & 5 & 2 & 3 & yes\\
\midrule
$(be,\,bf,\,cde,\,acd,\,adf,\,acef)$ & 10 & 6 & 3 & 3 & yes\\
$(bf,\,cd,\,aef,\,abe)$ & 10 & 4 & 1 & 3 & yes\\
$(cf,\,df,\,abe,\,acd,\,bde)$ & 10 & 5 & 2 & 3 & yes\\
$(af,\,be,\,bcd,\,cdf,\,acde)$ & 10 & 5 & 2 & 3 & yes\\
\midrule
$(cd,\,abe,\,abf,\,aef,\,bdf,\,cef)$ & 11 & 6 & 3 & 3 & yes\\
$(bf,\,abe,\,acd,\,ace,\,bcd,\,def)$ & 11 & 6 & 3 & 3 & yes\\
\midrule
$(abf,\,ace,\,acf,\,ade,\,bcd,\,bde,\,bef,\,cdf)$ & 12 & 8 & 5 & 3 & no\\
\end{longtable}

\begin{longtable}{@{}l c c c c c@{}}
\caption{
\cite{RTY2020}*{Theorem 4.1}: deviation, maximal socle degree, and licci status.}
\label{tab:RTY2020-thm41}\\
\toprule
\multicolumn{1}{c}{$I$} & $f_3$ & $\mu(I)$ & $d(I)$ & socle deg. & licci\\
\midrule
\endfirsthead
\toprule
\multicolumn{1}{c}{$I$} & $f_3$ & $\mu(I)$ & $d(I)$ & socle deg. & licci\\
\midrule
\endhead
\bottomrule
\endfoot

$(cf,\,cg,\,fg,\,abde)$ & 12 & 4 & 1 & 4 & yes\\
$(bg,\,cf,\,cg,\,adef,\,abcde)$ & 12 & 5 & 2 & 4 & yes\\
\midrule
$(cg,\,dg,\,bcf,\,ade,\,abefg)$ & 13 & 5 & 2 & 4 & yes\\
\midrule
$(cf,\,de,\,abdg,\,abeg)$ & 14 & 4 & 1 & 4 & yes\\
$(cg,\,dg,\,cdf,\,abef,\,abde)$ & 14 & 5 & 2 & 4 & yes\\
$(cg,\,dg,\,cef,\,abef,\,abde,\,abcdf)$ & 14 & 6 & 3 & 4 & yes\\
$(cf,\,cg,\,aeg,\,abde,\,bdfg)$ & 14 & 5 & 2 & 4 & yes\\
$(bg,\,cf,\,acde,\,adeg,\,abdef)$ & 14 & 5 & 2 & 4 & yes\\
\midrule
$(cg,\,acf,\,bde,\,deg,\,abfg)$ & 15 & 5 & 2 & 4 & yes\\
$(ef,\,acg,\,bdg,\,acf,\,abcde)$ & 15 & 5 & 2 & 4 & yes\\
$(cf,\,aeg,\,ceg,\,bdf,\,abdg,\,abcde)$ & 15 & 6 & 3 & 4 & yes\\
$(de,\,dg,\,bcfg,\,abcf,\,acef,\,abeg)$ & 15 & 6 & 3 & 4 & yes\\
$(bg,\,bcf,\,ade,\,cfg)$ & 15 & 4 & 1 & 4 & yes\\
\midrule
$(de,\,acf,\,bdg,\,abeg,\,bcfg)$ & 16 & 5 & 2 & 4 & yes\\
$(cf,\,bdg,\,cdg,\,aefg,\,abde,\,abeg)$ & 16 & 6 & 3 & 4 & yes\\
$(acf,\,bcf,\,deg,\,bde,\,aeg,\,bdf,\,abcdg)$ & 16 & 7 & 4 & 4 & yes\\
$(bg,\,acf,\,cfg,\,acde,\,adeg,\,bdef)$ & 16 & 6 & 3 & 4 & yes\\
$(cf,\,acg,\,abf,\,bdeg,\,bcde,\,adeg)$ & 16 & 6 & 3 & 4 & yes\\
$(bg,\,dfg,\,bcf,\,adef,\,acef,\,acde)$ & 16 & 6 & 3 & 4 & yes\\
$(cg,\,ade,\,bef,\,abdf)$ & 16 & 4 & 1 & 4 & yes\\
\midrule
$(acf,\,beg,\,bde,\,bdg,\,ceg,\,adef)$ & 17 & 6 & 3 & 4 & yes\\
$(acf,\,ade,\,def,\,bcg,\,bfg)$ & 17 & 5 & 2 & 4 & yes\\
$(adg,\,bdf,\,bdg,\,bef,\,cef,\,abcg,\,acde)$ & 17 & 7 & 4 & 4 & yes\\
$(acf,\,acg,\,bdf,\,ceg,\,efg,\,abde)$ & 17 & 6 & 3 & 4 & yes\\
$(acf,\,bcf,\,bcg,\,bdg,\,deg,\,abde,\,abef)$ & 17 & 7 & 4 & 4 & yes\\
$(cf,\,aeg,\,abde,\,abdg,\,bcdg,\,bdef)$ & 17 & 6 & 3 & 4 & yes\\
$(bg,\,bcf,\,acde,\,acdf,\,adef,\,aefg,\,cdeg)$ & 17 & 7 & 4 & 4 & yes\\
$(acf,\,acg,\,bde,\,bdf,\,beg,\,defg)$ & 17 & 6 & 3 & 4 & yes\\
\midrule
$(ade,\,bcf,\,bcg,\,bfg,\,aceg,\,adfg,\,cdef)$ & 18 & 7 & 4 & 4 & yes\\
$(ade,\,acf,\,bcg,\,bfg,\,bdef,\,cdeg)$ & 18 & 6 & 3 & 4 & yes\\
$(acg,\,bde,\,bef,\,cdg,\,abcf,\,adef,\,adfg)$ & 18 & 7 & 4 & 4 & yes\\
$(adf,\,adg,\,aeg,\,bcf,\,bcde,\,bdeg,\,cefg)$ & 18 & 7 & 4 & 4 & yes\\
$(acg,\,ade,\,adg,\,cde,\,abcf,\,bcef,\,bdfg,\,befg)$ & 18 & 8 & 5 & 4 & no\\
$(ade,\,beg,\,cdg,\,deg,\,abcf,\,abdf,\,acfg,\,bcef)$ & 18 & 8 & 5 & 4 & no\\
$(bde,\,bdf,\,beg,\,ceg,\,abcf,\,acdg,\,acef,\,adfg)$ & 18 & 8 & 5 & 4 & no\\
\end{longtable}

\renewcommand{\arraystretch}{1.5}
\begin{longtable}{@{}c l c c c c c@{}}
\caption{\cite{KTY2013}*{Theorem 3.7}: deviation and maximal socle degree.
Here $g=\operatorname{ht}I(G)$.}
\label{tab:KTY2013-thm37}\\
\toprule
$G$ & \multicolumn{1}{c}{$I(G)$} & $|V|$ & $\mu(I)$ & $d(I)$ & socle deg. & range\\
\midrule
\endfirsthead
\toprule
$G$ & \multicolumn{1}{c}{$I(G)$} & $|V|$ & $\mu(I)$ & $d(I)$ & socle deg. & range\\
\midrule
\endhead
\bottomrule
\endfoot

$G_a$ & $\{x_iy_i\}_{1\le i\le g}$ & $2g$ & $g$ & $0$ & $g$ & $g\ge 1$\\
$G_b$ & $\{x_iy_i\}_{1\le i\le g}\cup\{x_jy_1\}_{2\le j\le h}$ & $2g$ & $g+h-1$ & $h-1$ & $g-1$ & $2\le h\le g$\\
$G_{c_1}$ & $\{x_iy_i\}_{1\le i\le g-2}\cup\{x_jz_1\}_{1\le j\le h}\cup\{z_1z_2,z_2z_3,z_3z_1\}$ & $2g-1$ & $g+h+1$ & $h+1$ & $g-1$ & $0\le h\le g-2$\\
$G_{c_2}$ & $\{x_iy_i\}_{1\le i\le g-3}\cup\{z_1z_2,z_2z_3,z_3z_4,z_4z_5,z_5z_1\}$ & $2g-1$ & $g+2$ & $2$ & $g-1$ & $g\ge 3$\\
\end{longtable}

\renewcommand{\arraystretch}{1.4}
\begin{longtable}{@{}l c c c c c@{}}
\caption{\cite{ERT2020}*{Theorem 3.5}: deviation and maximal socle degree.
Here $\mu(J_G)=|E(G)|$ and $\operatorname{ht}J_G=n-1$.}
\label{tab:ERT2020-thm35}\\
\toprule
$G$ & $|V|$ & $\mu(J_G)$ & $d(J_G)$ & socle deg. & range\\
\midrule
\endfirsthead
\toprule
$G$ & $|V|$ & $\mu(J_G)$ & $d(J_G)$ & socle deg. & range\\
\midrule
\endhead
\bottomrule
\endfoot

path graph $P_n$ & $n$ & $n-1$ & $0$ & $n-1$ & $n\ge 1$\\
triangle with three paths (Fig.~3) & $n=3+r+s+t$ & $n$ & $1$ & $n-2$ & $r,s,t\ge 0$\\
\end{longtable}

\vspace{.15cm}

\bibliographystyle{acm}
\bibliography{references}

\end{document}